\documentclass[onecolumn, 12pt]{IEEEtran}
\usepackage{amsmath,amsthm,amssymb,amsfonts,enumerate,mathtools}
\usepackage{graphs}

\theoremstyle{plain}
\newtheorem{theorem}{Theorem}[section]
\newtheorem{lemma}[theorem]{Lemma}

\newtheorem{prop}[theorem]{Proposition}
\newtheorem{corollary}[theorem]{Corollary}
\newtheorem{defn}[theorem]{Definition}
\newtheorem{example}[theorem]{Example}  
\newtheorem{remark}[theorem]{Remark}

\newcommand{\mc}[1]{\mathcal{#1}}
\newcommand{\vv}[1]{\mathbf{#1}}

\def\supp{{\operatorname{supp}}}

\newcommand\vlambda{{\boldsymbol{\lambda}}}
\newcommand\vomega{{\boldsymbol{\omega}}}
\newcommand\vepsilon{{\boldsymbol{\epsilon}}}
\newcommand\bbr{\mathbb R}
\newcommand\R{\mathbb R}

\newcommand\Z{\mathbb Z}

\begin{document}

\title{LP Pseudocodewords of Cycle Codes are Half-Integral}

\author{Nathan Axvig\thanks{This work was supported in part by NSF grants DMS-0602332 and DMS-0838463.}\thanks{N. Axvig is with the Department
  of Applied Mathematics, Virginia Military Institute, Lexington, VA 24450, USA.}\thanks{N. Axvig's email:  axvignd10@vmi.edu}
}

\maketitle


\begin{abstract}
In his Ph.D. disseration, Feldman and his collaborators define the linear programming decoder for binary linear codes, which is a linear programming relaxation of the maximum-likelihood decoding problem.  This decoder does not, in general, attain maximum-likelihood performance; however, the source of this discrepancy is known to be the presence of non-integral extreme points (vertices) within the fundamental polytope, vectors which are also called nontrivial linear programming pseudocodewords.  Restricting to the class of cycle codes, we provide necessary conditions for a vector to be a linear programming pseudocodeword.  In particular, the components of any such pseudocodeword can only assume values of zero, one-half, or one.
\end{abstract}


\maketitle

\section{Introduction}\label{sec:introduction}

Let $C$ be a binary linear code of length $n$, and consider transmission over a binary-input, memoryless channel.  Under the the additional assumptions that the channel is symmetric and that each codeword is equally likely to be transmitted, a \emph{maximum-likelihood decoder} returns a codeword that minimizes the probability of word-error.  It is well-known (see, e.g.,~\cite{axvig10, Fel00, klvw2}) that the set of maximum-likelihood (ML) codewords coincides with the solution set of the following integer program:
\begin{equation}\label{eqn:integerprogram}
\begin{array}{rll}
\text{minimize} & \vlambda^T \vv c  \\
\text{subject to} & \vv c \in C
\end{array}
\end{equation}
where $\vlambda$ is the vector of log-likelihood ratios based on the channel output.  While solutions to integer linear programs such as Problem~(\ref{eqn:integerprogram}) are often of great interest, solving integer linear programs in general is quite difficult.  On the other hand, solutions to linear programs (problems in which the output vector is not constrained to be an integer vector) can often be found quickly (using, e.g., the simplex algorithm), and they can even be found in polynomial time (using, e.g., the ellipsoid method)~\cite{bt}.  A common approach to solving integer linear programs therefore is to solve a related linear program (or a series of linear programs) and then interpret the solution in light of the original problem.

Letting $\text{conv}(C)$ be the convex hull of $C$ in $\R^n$ (where we regard codewords as vectors in $\{0,1\}^n \subseteq \R^n$), Problem~(\ref{eqn:integerprogram}) is equivalent to
\begin{equation}\label{eqn:MLproblem}
\begin{array}{rll}
\text{minimize} & \vlambda^T \vv f  \\
\text{subject to} & \vv f \in \text{conv}(C).
\end{array}
\end{equation}
Problem~(\ref{eqn:MLproblem}) is a linear program; however, the constraints needed to define the feasible region $\text{conv}(C)$ explicitly are likely to be either (a) difficult to compute or (b) very large in number.  Elaborating on this, suppose that there was a polynomial time algorithm that was capable of taking a general binary linear code $C$ and producing a set of constraints describing $\text{conv}(C)$.  This algorithm, along with the ellipsoid method, can then be used to solve instances of the COSET WEIGHTS decision problem (see~\cite{berlekamp} for a definition) in polynomial time.  Since COSET WEIGHTS is NP-complete~\cite{berlekamp}, this would imply P = NP.

%

In an attempt to find solutions to Problem~(\ref{eqn:MLproblem}) efficiently, Feldman~\cite{Fel00} defines the \emph{linear programming (LP) decoder}.  The linear programming decoder returns as its output the solution to a \emph{linear programming relaxation} of Problem~(\ref{eqn:MLproblem}), namely
\begin{equation}\label{eqn:LPproblem}
\begin{array}{rll}
\text{minimize} & \vlambda^T \vv f  \\
\text{subject to} & \vv f \in \mc P.
\end{array}
\end{equation}
The feasible set $\mc P$ of this linear program is known as the \emph{fundamental polytope}.
The fundamental polytope has several properties that make it a reasonable choice for a relaxation of Problem~(\ref{eqn:MLproblem}), not the least of which is that, from a computational complexity standpoint, it is easier to describe than $\text{conv}(C)$.  More specifically, if we assume that the row-weight of the parity-check matrix defining the code is bounded by a constant, then the number of constraints that describe $\mc P$ is bounded by a polynomial in $n$, where $n$ is the length of the code (see Section~\ref{sec:background_fundamental_polytope}).

The fundamental polytope is a subset of the unit hypercube $[0,1]^n$ and thus is bounded. Because $\mc P$ is bounded, the general theory of linear programming states that (a) a solution to Problem~(\ref{eqn:LPproblem}) always exists and that (b) we may assume that at least one solution is an extreme point\footnote{Extreme points are often referred to as vertices in linear programming literature.  We refrain from applying the term ``vertex'' to polyhedra to avoid confusion when speaking of both graphs and polyhedra simultaneously.  Thus, the terms ``node'' and ``vertex'' are reserved for graphs, and ``extreme point'' is reserved for polyhedra.} of the underlying polytope~\cite{bt}.  We therefore make the convention that the LP decoder always returns an extreme point of the fundamental polytope.  As such, we say that any extreme point of $\mc P$ is a \emph{linear programming (LP) pseudocodeword}.  Linear programming pseudocodewords are the principal objects of investigation in this paper.

In~\cite{Fel00}, it shown that the set of integer extreme points of $\mc P$ is precisely the set of codewords $C$.  Such integer vectors are therefore called \emph{trivial linear programming (LP) pseudocodewords}.  Feldman~\cite{Fel00} uses this relationship between $C$ and $\mc P$ to show that the LP decoder has the \emph{ML-certificate property}:  if an optimal solution $\vv f^\ast$ to Problem~(\ref{eqn:LPproblem}) is a integer vector, then $\vv f^\ast$ must be an ML codeword.  It is often the case, however, that the fundamental polytope contains non-integer extreme points, known as \emph{nontrivial linear programming (LP) pseudocodewords}, in addition to the integer-valued codewords.  By the ML-certificate property, it is precisely the presence of nontrivial LP pseudocodewords that prevents the LP decoder from attaining the performance (with respect to word-error rate) of the ML decoder.

In this paper, we provide new necessary conditions for a vector to be a linear programming pseudocodeword of a \emph{cycle code}.  In particular, we show that any LP pseudocodeword of a cycle code must be half-integral; that is, any LP pseudocodeword must be an element of $\{0,\frac 12, 1\}^n$.  Moreover, in proving this half-integrality we also discover that LP pseudocodewords of cycle codes display additional structure that can be stated succinctly in terms of the code's \emph{Tanner graph}.  The results of this paper are applied in~\cite{jointpaper2010}, where the author and his collaborator provide complete graphical characterizations of both LP pseudocodewords and \emph{minimal linear programming pseudocodewords} (see~\cite{VonKoe06}) for the class of cycle codes.

The remainder of this paper is organized as follows.  Section~\ref{sec:background}, the background section, is broken into three parts.  In Section~\ref{sec:background_coding}, we introduce relevant coding theory terminology and the family of cycle codes.  Section~\ref{sec:background_extreme_points} reviews some basic definitions and facts about extreme points, and Section~\ref{sec:background_fundamental_polytope} gives an explicit description of the fundamental polytope.  Section~\ref{sec:halfintegral} makes up the mathematical bulk of this paper, and it is divided into two parts.    Section~\ref{sec:technicallemma} is dedicated to proving a technical result that is then used in Section~\ref{sec:proof} to prove Theorem~\ref{thm:halfintegral}, which is the main result of this work.  We make concluding remarks in Section~\ref{sec:conclusion}.

\section{Background}\label{sec:background}

\subsection{Background on Coding Theory}\label{sec:background_coding}

The results herein require knowledge not only of a code itself but also of the specific parity-check presenting the code.  By ``code,'' we therefore mean a binary linear code equipped with a fixed parity-check matrix.\footnote{This is a departure in terminology from many standard texts and publications in coding theory where a code is described intrinsically as a subspace, not extrinsically as the kernel of a specific matrix.}

\begin{defn}\label{defn:Tannergraph}
A \emph{Tanner graph} is a bipartite graph $G$ with vertex set $X \cup U$ and edge set $E$.  The elements of $X$ are called \emph{variable nodes}, and the elements of $U$ are called \emph{check nodes}.
\end{defn}

Given an $r \times n$ parity-check matrix $H = (h_{j,i})$, one can construct the Tanner graph $G = G(H)$ of $H$ as follows:  set $X = \{x_1, x_2, \dots, x_n\}$ and  $U = \{ u_1, u_2, \dots u_r\}$. Define $G = G(H)$ to be the graph whose vertex set is $X \cup U$ with $x_i$ adjacent to $u_j$ if and only if $h_{j,i} = 1$.  In other words, $G(H)$ is the is the bipartite graph whose bipartite adjacency matrix is $H$.  Conversely, one can derive a parity-check matrix from a Tanner graph:  index rows by vertices in $U$ and columns by vertices in $X$.  Set $h_{u,x} = 1$ precisely when $x$ and $u$ are adjacent in $G$.  In this manner one can see that there is a bijective correspondence between Tanner graphs and parity-check matrices.

The Tanner graph is the structure on which the class of \emph{message-passing decoding algorithms}, e.g., the \emph{min-sum algorithm} or the \emph{sum-product algorithm}, operate.  Loosely speaking, message-passing decoding algorithms work by relaying bit and/or reliability information across edges of the Tanner graph in an iterative fashion.  For the purposes of this paper, however, the Tanner graph is introduced to provide a graphical environment in which to view codewords.

\begin{defn}\label{defn:validconfiguration}
Let $G = (X \cup U, E)$ be a Tanner graph, and let $\vv f$ be an assignment of binary values to the variable nodes of $G$.  The assignment $\vv f$ is called a \emph{valid configuration} on $G$ provided that for each $u \in U$ the quantity $\sum_{x \in N(u)} x_i$ is even, where, as is standard, $N(u)$ denotes the neighborhood of $u$ in $G$, i.e., the set of all $x \in X$ that are adjacent to $u$.
\end{defn}

By identifying valid configurations on a Tanner graph with vectors in $\{0,1\}^n$, we obtain a natural correspondence between valid configurations and codewords:

\begin{prop}[see, e.g.,~\cite{AMPTPW3}]\label{prop:valid}
Let $C$ be a code with parity-check matrix $H$, and let $G = G(H) = ( X \cup U, E)$ be the corresponding Tanner graph.  The set of valid configurations on $G$ corresponds precisely to the set $C$ of codewords.
\end{prop}

By Proposition~\ref{prop:valid}, codewords of $C$ may be viewed as valid binary configurations on a Tanner graph.  This graphical realization of codewords is further displayed in the family of \emph{cycle codes}.

\begin{defn}\label{defn:cyclecode}
A \emph{cycle code} is a code $C$ equipped with a parity-check matrix $H$ that has uniform column weight 2.
\end{defn}

We focus on cycle codes because they provide for a more fruitful analysis than general codes.  This is due to their special structure:  let $C$ be a cycle code with Tanner graph $G$.  If $\vv c \in C$, Proposition~\ref{prop:valid} states $\vv c$ must be a valid configuration on $G$.  Since every variable node of $G$ has degree two, the subgraph of $G$ induced by $\supp(\vv c) \cup N(\supp(\vv c))$ is even, i.e., all nodes in the subgraph have even degree (herein, $\supp(\vv v)$ is used to denote the \emph{support} of $\vv v$:  the set of positions where $\vv v$ is not zero).  Thus, the subgraph induced by $\supp(\vv c) \cup N(\supp(\vv c))$ must be a union of edge-disjoint simple cycles.  The converse is also true:  the indicator vector for a set of variable nodes within a union of edge-disjoint simple cycles is a codeword.

\begin{example}\label{ex:cyclecodeword}
The parity-check matrix
\[
H = \begin{bmatrix}
1 & 1 & 0 & 1 & 0 & 0 & 0 & 0 & 0 \\
0 & 1 & 1 & 0 & 1 & 0 & 0 & 0 & 0 \\
1 & 0 & 1 & 0 & 0 & 1 & 0 & 0  & 0\\
0 & 0 & 0 & 1 & 0 & 0 & 1 & 0 & 1 \\
0 & 0 & 0 & 0 & 1 & 0 & 1 & 1 & 0\\
0 & 0 & 0 & 0 & 0 & 1 & 0 & 1 & 1 \\
\end{bmatrix}
\]
has uniform column weight two, and hence defines a cycle code.  One can check that the vector
\[
\vv c = (1,1,0,0,1,1,0,1,0)^T
\]
is in the null-space of $H$.  It also gives rise to a simple cycle in the Tanner graph of the code; see Figure~\ref{fig:cyclecodeword} for an illustration.  We note that an isomorphic Tanner graph appears in Figure 6.1 of~\cite{Wib96}.
\end{example}

\begin{figure}
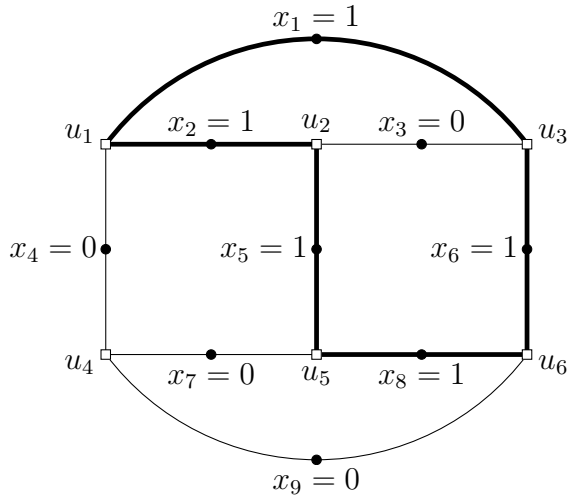

\begin{center}
$
\begin{graph}(8,6.5)
\unitlength=.7cm

\squarenode{f1}(2,7)[\graphnodecolour(1,1,1)]
\squarenode{f2}(6,7)[\graphnodecolour(1,1,1)]
\squarenode{f3}(10,7)[\graphnodecolour(1,1,1)]
\squarenode{f4}(2,3)[\graphnodecolour(1,1,1)]
\squarenode{f5}(6,3)[\graphnodecolour(1,1,1)]
\squarenode{f6}(10,3)[\graphnodecolour(1,1,1)]


\roundnode{x1}(6,9)
\roundnode{x2}(4,7)
\roundnode{x3}(8,7)
\roundnode{x4}(2,5)
\roundnode{x5}(6,5)
\roundnode{x6}(10,5)
\roundnode{x7}(4,3)
\roundnode{x8}(8,3)
\roundnode{x9}(6,1)

\bow{f1}{f3}{.25}[\graphlinewidth{0.09}]
\edge{f1}{f2}[\graphlinewidth{0.09}]
\edge{f2}{f3}
\edge{f1}{f4}
\edge{f2}{f5}[\graphlinewidth{0.09}]
\edge{f3}{f6}[\graphlinewidth{0.09}]
\edge{f4}{f5}
\edge{f5}{f6}[\graphlinewidth{0.09}]
\bow{f4}{f6}{-.25}

\nodetext{x1}(0,.4){$x_1 = 1$}
\nodetext{x2}(0,.4){$x_2 = 1$}
\nodetext{x3}(0,.4){$x_3 = 0$}
\nodetext{x4}(-1,0){$x_4 = 0$}
\nodetext{x5}(-1,0){$x_5 = 1$}
\nodetext{x6}(-1,0){$x_6 = 1$}
\nodetext{x7}(0,-.4){$x_7 = 0$}
\nodetext{x8}(0,-.4){$x_8 = 1$}
\nodetext{x9}(0,-.4){$x_9 = 0$}

\nodetext{f1}(-.5,.2){$u_1$}
\nodetext{f2}(0,0.4){$u_2$}
\nodetext{f3}(.5,.2){$u_3$}
\nodetext{f4}(-.5,-.2){$u_4$}
\nodetext{f5}(0,-.4){$u_5$}
\nodetext{f6}(.5,-.2){$u_6$}

\end{graph}
$
\hspace{1cm}
\end{center}
\caption{The configuration corresponding to the codeword $\vv c = (1,1,0,0,1,1,0,1,0)^T$ of Example~\ref{ex:cyclecodeword}.}
\label{fig:cyclecodeword}
\end{figure}

\subsection{Background on Linear Programming and Extreme Points}\label{sec:background_extreme_points}

Most of the material in this section can be found in~\cite{bt}.  For our purposes, we define a \emph{linear program} to be an optimization problem of the form
\begin{equation*}
\begin{array}{rll}
\text{minimize} & \vv q^T \vv f &   \\
\text{subject to} &  \vv a_j^T \vv f \geq  b_j & j \in M_1 \\
 &  \vv a_j^T \vv f \leq  b_j & j \in M_2 \\
  &  \vv a_j^T \vv f = b_j & j \in M_3 \\
& f_i \geq 0 &  i \in N_1 \\
& f_i \leq 0 &  i \in N_2,
\end{array}
\end{equation*}
where $\vv q$, $\vv f$, and the $\vv a_j$'s are $n$-dimensional real vectors and the index sets $M_1, M_2, M_3, N_1,$ and $N_2$ are finite.\footnote{In general, a linear program can be either a maximization or a minimization problem.  There is, however, no essential loss of generality in using this convention:  maximizing $\vv q^T \vv f$ is equivalent to minimizing $(-\vv q)^T \vv f$.}  The vector $\vv q$ is called the \emph{cost function} of the linear program.  The linear functional $\vv q^T \vv f$ is the \emph{objective function} or the \emph{cost function}.  The vectors $\vv a_j$ are called \emph{constraint vectors}, or simply \emph{constraints}.  Also called constraints are the (in)equalities of the form $\vv a_j^T \vv f \geq  b_j, \vv a_j^T \vv f \leq  b_j,$ or $\vv a_j^T \vv f =  b_j$.

Any vector $\vv f$ that satisfies all of the linear program's constraints is called a \emph{feasible solution}, and the set of all feasible solutions is the \emph{feasible set}.  A feasible solution that minimizes the objective function is an \emph{optimal feasible solution}, or simply an \emph{optimal solution}.  Note that while the optimality of a vector $\vv f$ depends on both the constraints and the cost function, the feasibility of a solution $\vv f$ depends only on the constraints.  Since the feasible set for a linear program is an intersection of closed half-spaces in $\bbr^n$, it is a \emph{polyhedron}.  A polyhedron that is also bounded is called a \emph{polytope}.  When dealing with polyhedra in the context of linear programming, particular attention is paid to the \emph{extreme points}.

\begin{defn}[\cite{bt}, Definition 2.6]\label{defn:extremepoint}
Let $\mc M$ be a nonempty polyhedron in $\R^n$.  A point $\vomega \in \mc M$ is an \emph{extreme point} provided that it is not in the convex hull of $\mc M \setminus \{\vomega\}$.  In other words, $\vomega$ is an extreme point of $\mc M$ if it cannot be written as a convex sum $\alpha \vv y + (1-\alpha) \vv z$ for $\vv y, \vv z \in \mc M \setminus \{ \vomega\}$ and $\alpha \in (0,1)$.
\end{defn}

The importance of extreme points is summarized in the next two theorems.

\begin{theorem}[\cite{bt}, Theorem 2.7]\label{thm:lpoptimal}
Suppose that a linear program has feasible set $\mc M$.  If $\mc M$ has an extreme point and an optimal solution to the linear program exists, then there exists an optimal solution that is an extreme point of $\mc M$.
\end{theorem}

Theorem~\ref{thm:lpoptimal} implies that we may assume that an optimal solution to a linear program occurs at an extreme point.  In particular, when solving a linear program whose underlying polyhedron is bounded, an optimal solution is guaranteed to exist and it therefore suffices to consider only extreme points as candidate solutions. On the other hand, given an extreme point $\vv e$ of a polyhedron $\mc M$ there is always some linear program with $\mc M$ as its feasible set whose unique solution is $\vv e$:

\begin{theorem}[\cite{bt}, Theorem 2.3]\label{thm:extremepointsolution}
Let $\mc M$ be a polyhedron in $\R^n$.  If $\vomega$ is an extreme point of $\mc M$, then there exists a cost function $\vv q \in \bbr^n$ such that $\vomega$ is the unique optimal solution to the following linear program:
\[\begin{array}{rl}
\text{minimize} & \vv q^T \vv f    \\
\text{subject to} &  \vv f \in \mc M.
\end{array}\]
In other words, $\vv q^T \vomega < \vv q^T \vv f$ for all $\vv f \in \mc M \setminus \{ \vomega \}$.
\end{theorem}

It is often convenient to represent extreme points algebraically instead of geometrically.  We say that a constraint of the form $\vv a^T \vv f \geq  b, \vv a^T \vv f \leq  b$, or $\vv a^T \vv f =  b$ is \emph{active} at $\vomega$ if $\vv a^T \vomega = b$, i.e., if the constraint is met with equality.  If the polyhedron $\mc M$ is defined by linear equality and inequality constraints and $\vomega$ is an element of $\R^n$, then we say that $\vomega$ is a \emph{basic solution} if all equality constraints of $\mc M$ are satisfied and the set of active constraint vectors spans all of $\bbr^n$.  If $\vomega$ is a basic solution that satisfies all of the constraints, it is called a \emph{basic feasible solution}.  Using these algebraic notions, Theorem~\ref{thm:bfs} gives an alternate characterization of extreme points.

\begin{theorem}[\cite{bt}, Theorem 2.3]\label{thm:bfs}
Let $\mc M$ be a nonempty polyhedron in $\R^n$ defined by a given set of linear equality and inequality constraints, and let $\vomega \in \mc M$.  The following are equivalent:
\begin{enumerate}[(a)]
\item  $\vomega$ is an extreme point of $\mc M$,
\item $\vomega$ is a basic feasible solution of $\mc M$.
\end{enumerate}

\end{theorem}

Theorem~\ref{thm:bfs} relates a geometric invariant to the specific algebraic representation of a polyhedron.  It also points to a method of finding extreme points:  suppose that $\vv a_1, \dots, \vv a_n$ is a set of $n$ linearly independent constraint vectors, and let $b_1, \dots, b_n$ be the associated constraint values.  Let $A$ be the matrix whose rows are $\vv a_1, \dots, \vv a_n$, and let $\vv b$ be a column vector whose entries are $b_1, \dots, b_n$.  Since the rows of $A$ are linearly independent, there is a solution to $A \vv f = \vv b$, namely $\vv f = A^{-1} \vv b$.  By the construction of $\vv f$, it satisfies the constraints $\vv a_1, \dots, \vv a_n$ with equality.  If $\vv f$ turns out to satisfy the rest of the constraints of the polyhedron, then $\vv f$ is a basic feasible solution and hence, by Theorem~\ref{thm:bfs}, an extreme point.

Conversely, if $\vv f$ is an extreme point of a polyhedron, then Theorem~\ref{thm:bfs} implies that there exists a set of $n$ linearly independent constraint vectors that are active at $\vv f$; call them $\vv a_1, \dots, \vv a_n$.  With $A$ and $\vv b$ as in the preceding paragraph, the extreme point $\vv f$ must be the unique solution to $A \vv f = \vv b$.  This interpretation of an extreme point as a solution to a linear system of equations gives the proof of the following well-known theorem.

\begin{theorem}\label{thm:rational}
Let $\mc M$ be a non-empty polyhedron in $\R^n$.  If every constraint vector defining $\mc M$ has rational coefficients, then every extreme point of $\mc M$ is a rational vector.
\end{theorem}

\subsection{The Fundamental Polytope}\label{sec:background_fundamental_polytope}

We now explicitly describe the fundamental polytope, which is the feasible region over which Feldman's linear programming decoder operates (see Problem~(\ref{eqn:LPproblem})).

\begin{defn}[\cite{Fel00}]\label{defn:fundamentalpolytope}
Let $C$ be a code with parity-check matrix $H$ and Tanner graph $G= (X \cup U, E)$.  The \emph{fundamental polytope} $\mc P = \mc P(H)$ is the set of all vectors $\vv f \in \bbr^{|X|}$ satisfying the following constraints:
\begin{enumerate}[(a)]
\item $0 \leq f_x \leq 1$ for all $x \in X$, and
\vspace{.1cm}
\item $\displaystyle \sum_{x \in S} f_x + \sum_{x \in N(u)\setminus S}(1- f_x) \leq |N(u)| - 1$ for all pairs $(u, S)$, where  $u \in U $ and $S$ is a subset of $N(u)$ with odd cardinality.
\end{enumerate}
\end{defn}

Feldman~\cite{Fel00} offers a heuristic justification for the constraints defining $\mc P$, which we reproduce here for the sake of clarity.  Since all codewords lie in $\{0,1\}^n$, the components of a vector in $\mc P$ should not be able to assume values outside of $[0,1]$, but since linear programs require convexity we permit these variables to assume intermediate values.  The other nontrivial family of constraints is best explained in the integer case, for then these constraints can be viewed as cutting planes which forbid ``bad'' configurations.  To be explicit, suppose that $\vv f \in \{0,1\}^n$ is not a codeword.  This means that there exists a row $\vv r_j$ of the parity-check matrix $H$ such that $\vv r_j \cdot \vv f = 1$, which in turn implies that there is a check node $j \in \mc J$ of the Tanner graph $T$ such that $\vv f$ assigns an odd number of ones to the variable nodes of $N(j)$ in $T$.  Letting $S$ denote this subset, we see that $\vv f$ is such that
\begin{align*}
\sum_{i \in S} f_i + \sum_{i \in N(j)\setminus S}(1- f_i)  & = \sum_{i \in S} 1 + \sum_{i \in N(j)\setminus S} 1 \\
& = |S| + |N(j) \setminus S| \\
& = |N(j)| \\
& > |N(j)| -1.
\end{align*}
Hence, the non-codeword integer vector $\vv f$ is excluded from $\mc P$.  Conversely, if $\vv f \in \{0,1\}^n$ is a codeword, then one can check that it satisfies every constraint of the fundamental polytope.  These two results are summarized in Theorem~\ref{thm:proper} below.

\begin{theorem}[\cite{Fel00}]\label{thm:proper}
Let $C$ be a code with parity-check matrix $H$ and fundamental polytope $\mc P = \mc P(H)$.  An integer vector $\vv f$ is an element of $\mc P$ if and only if $\vv f \in C$.
\end{theorem}

Since all the constraints of the fundamental polytope are integer vectors, Theorem~\ref{thm:rational} states that all extreme points of the fundamental polytope must be rational vectors.  If an extreme point $\vv f$ is an integer vector, Theorem~\ref{thm:proper} states that $\vv f$ must be a codeword.  As mentioned in the introduction, there are often fractional extreme points of $\mc P$ as well -- see Example~\ref{ex:pcws}.  In Section~\ref{sec:halfintegral}, we investigate the structure of the non-integral extreme points.

\begin{example}\label{ex:pcws}
Let $C$ be the cycle code whose parity-check matrix is given in Example~\ref{ex:cyclecodeword} and whose Tanner graph is depicted in Figure~\ref{fig:cyclecodeword}, and let $\mc P$ be the fundamental polytope of $C$.  Theorem~\ref{thm:proper} states that the set of integer extreme points of $\mc P$ is precisely $C$.  Since the dimension of $C$ is $k = 4$, it follows that $\mc P$ contains $2^4 = 16$ integer extreme points.  As is shown in~\cite{jointpaper2010} (and also in~\cite{axvig10}), there are exactly four non-integral extreme points of $\mc P$:
\begin{center}
\begin{tabular}{c}
$\left(\frac 12, \frac 12, \frac 12, 1, 0, 0, \frac 12, \frac 12, \frac 12\right)$\phantom{.}  \vspace{.2cm} \\ $\left(\frac 12, \frac 12, \frac 12, 0, 1, 0, \frac 12, \frac 12, \frac 12\right)$\phantom{.} \vspace{.2cm} \\ 
$\left(\frac 12, \frac 12, \frac 12, 0, 0, 1, \frac 12, \frac 12, \frac 12\right)$\phantom{.} \vspace{.2cm} \\ \hspace{-.25cm} $\left(\frac 12, \frac 12, \frac 12, 1, 1, 1, \frac 12, \frac 12, \frac 12\right)$.\\
\end{tabular}
\end{center}
\end{example}
\section{LP Pseudocodewords of Cycle Codes are Half-Integral}\label{sec:halfintegral}

This section is dedicated to proving that linear programming pseudocodewords are half-integral.  Along the way, we will see that LP pseudocodewords also have a close relationship with unions of vertex-disjoint simple cycles in the Tanner graph.  This additional structure turns out to be inseparable from half-integrality.  We make this latter fact explicit in Theorem~\ref{thm:halfintegral}, which is the main result of this paper.

In order to prove Theorem~\ref{thm:halfintegral}, we adopt the viewpoint suggested by Theorem~\ref{thm:bfs} -- namely, we view an extreme point of a polytope as a solution to a linear system of equations.  In our case, the equations forming this system come from the constraint vectors of the fundamental polytope that are satisfied with equality.  We proceed as follows.  We begin by prove a technical lemma pertaining to active constraints in Section~\ref{sec:technicallemma}.  This lemma is then applied to cycle codes\footnote{While this paper's main result concerns linear programming pseudocodewords of cycle codes, Lemma~\ref{lemma:almostunique} can be applied to any binary linear code.} in Section~\ref{sec:proof}, giving a proof of Theorem~\ref{thm:halfintegral}.

\subsection{A Technical Lemma on Active Constraints}\label{sec:technicallemma}

Let $C$ be a code defined by parity-check matrix $H$, and let $G = (X \cup U, E)$ be its Tanner graph.  We may rearrange the constraints defining the fundemental polytope to put them into the standard form presented in Section~\ref{sec:background_extreme_points}:
\begin{eqnarray}
f_x  \leq &1 & \forall x \in X \label{eqn:lessthan1}\\
f_x  \geq &0 & \forall x \in X \label{eqn:greaterthan0} \\
\sum_{x \in S} f_x + \sum_{x \in N(u)\setminus S}- f_x  \leq & |S|- 1 &\begin{tabular}{c} $\forall u \in U, S\subseteq N(u)$ \\  $\text{ such that $|S|$ is odd}$\label{eqn:checkconstraints} \end{tabular}
\end{eqnarray}
For convenience, we use the notation $(u,S)$ to represent the constraint vector of $\mc P$ corresponding to the check $u$ and the odd-sized subset $S \subseteq N(u)$.

\begin{defn}\label{defn:active}
Let $C$ be a code presented by a parity-check matrix $H$.  Let $G =( X \cup U, E)$ be its Tanner graph, and let $\mc P$ be its fundamental polytope.  For a check node $u\in U$, a subset $S \subseteq N(u)$, and a vector $\vv f \in [0,1]^n$, the \emph{cost} of $S$ at $u$ relative to $\vv f$ is given by
\[
\kappa_{u,\vv f}(S) = \sum_{x \in S} f_x + \sum_{x \in N(u)\setminus S} (1-f_x).
\]
If there exists an odd-sized subset $S\subseteq N(u)$ such that $\kappa_{u,\vv f}(S) = |N(u)| -1$, then we say that the check $u$ is \emph{active} at $\vv f$ and that the set $S$ is \emph{active} for $\vv f$ at $u$.
\end{defn}
Note that a vector $\vv f \in [0,1]^n$ is in the fundamental polytope if and only if $\kappa_{u,\vv f} (S) \leq |N(u)| - 1$ for all pairs $(u, S)$ such that $S$ is an odd-sized subset of $N(u)$.  The next lemma follows directly from Definition~\ref{defn:active} in conjunction with the definition of the fundamental polytope.

\begin{lemma}\label{lemma:activehalves}
Let $C$ be a code presented by a parity-check matrix $H$.  Let $G = ( X \cup U, E)$ be its Tanner graph, and let $\mc P$ be its fundamental polytope. If $u \in U$ is active at some $\vv f \in \mc P$ and $\alpha \in (0,1)$, then $\vv f$ can assume the value of $\alpha$ in at most $\max \{  \frac 1 \alpha ,  \frac{1}{1-\alpha}  \}$ positions of $N(u)$.  In particular, $\vv f$ can assume the value of $\frac 12$ at most twice in $N(u)$.
\end{lemma}

\begin{proof}
Let $D\subseteq X$ be the set of variable nodes in $N(u)$ that are assigned a value of $\alpha$ by $\vv f$, and let $d := |D|$.   Since $u$ is active at $\vv f$, there exists an odd-sized subset $S \subseteq N(u)$ such that $S$ is active for $\vv f$ at $u$.  We consider two cases.

If $\alpha \geq \frac 12$, then $\alpha \geq 1- \alpha$.  We therefore have
\[
|N(u)| - 1 = \kappa_{u,\vv f}(S) \leq |N(u)| - d(1-\alpha),
\]
which implies that $d \leq \frac{1}{1-\alpha} = \max \{ \frac 1 \alpha, \frac{1}{1-\alpha}\}$.

On the other hand, if $\alpha < \frac 12$, then $\alpha < 1- \alpha$.  We therefore have
\[
|N(u)| - 1 = \kappa_{u,\vv f}(S) \leq |N(u)| - d\alpha,
\]
which implies that $d \leq \frac{1}{\alpha} = \max \{ \frac 1 \alpha, \frac{1}{1-\alpha}\}$.
\end{proof}

Recalling that the symmetric difference between two sets $S_1$ and $S_2$ is $S_1 \triangle S_2 := (S_1 \cup S_2)\setminus (S_1 \cap S_2)$, we now prove Lemma~\ref{lemma:almostunique}, which is crucial in establishing Theorem~\ref{thm:halfintegral}.

\begin{lemma}\label{lemma:almostunique}
Let $C$ be a code presented by a parity-check matrix $H$.  Let $G = (X \cup U, E)$ be its Tanner graph, and let $\mc P$ be its fundamental polytope.  Fix $\vv f \in \mc P$ and $u \in U$.  If $u$ is active at $\vv f$ and $\vv f$ is not integral on $N(u)$ (i.e., there is some $x \in N(u)$ such that $0< f_x < 1$), then there are at most two active sets for $\vv f$ at $u$.

Moreover, in the case that the number of active sets for $\vv f$ at $u$ is exactly two, the symmetric difference of these two sets consists of exactly two variable nodes, and these two variable nodes are exactly the neighbors of $u$ at which $\vv f$ is not integral.
\end{lemma}

\begin{proof}
Let $\vv f \in \mc P$ and $u \in U$ be given such that $\vv f^{N(u)}$ is not integral and $u$ is active at $\vv f$.  For each $x \in N(u)$, the value $f_x$ appears in the expression for $\kappa_{u,\vv f}(S)$ as $f_x$ if $x \in S$ or as $1-f_x$ if $x \in N(u) \setminus S$.  The basic idea of this proof is to address the following question:  how can $f_x$ make the largest contribution to $\kappa_{u,\vv f}(S)$?  It is clear that $f_x > 1-f_x$ if and only if $f_x > \frac 12$ and that $f_x < 1-f_x$ if and only if $f_x < \frac 12$.  From this, we see that $\kappa_{u,\vv f}(S)$ is maximized exactly when $S$ consists of all nodes $x$ such that $f_x > \frac 12$, and possibly some nodes $x$ with $f_x = \frac 12$.  In a search for sets $S$ that are active for $\vv f$ at $u$, however, we can only consider those sets $S$ whose cardinality is odd.  This parity-based issue is highly dependent on $\vv f$ itself, about which we know little.  We therefore consider several cases, each of which takes the greedy solutions that ignore parity and mashes them into solutions that respect this parity condition.

Define the following sets:
\begin{itemize}
\item $\mc L := \{ x \in N(u) \, : \, f_x < \frac 12 \}$
\item $\mc E := \{ x \in N(u) \, : \, f_x = \frac 12 \}$
\item $\mc G := \{ x \in N(u) \, : \, f_x > \frac 12 \}$
\item $\mc Q := \{ x \in N(u) \, : \,  | \frac 12 - f_x| \leq | \frac 12 - f_{x^\prime}| \text{ for all } x^\prime \in N(u) \}.$
\end{itemize}
The set $\mc Q$ contains all nodes $x$ such that $f_x$ is closest to $\frac 12$ among all $x^\prime \in N(u)$; note that $\mc Q = \mc E$ if $\mc E \neq \emptyset$.

Since $u$ is active at $\vv f$, by Lemma~\ref{lemma:activehalves} we know that $0 \leq |\mc E| \leq 2$.  Suppose first that $|\mc E| = 2$, and write $ \mc Q = \mc E =  \{q_1, q_2\}$.  If $|\mc G|$ is odd, then the sets $S_1 := \mc G$ and $S_2 := \mc G \cup {\mc Q}$ are the only two odd-sized subsets of $N(u)$ that maximize $\kappa_{u,\vv f}$ over all odd-sized subsets of $N(u)$.  Moreover, since $u$ is active at $\vv f$, the maximum value achieved by $\kappa_{u,\vv f}$ over all odd-sized subsets of $N(u)$ is precisely $|N(u)| -1$.  Having $|\mc E| = 2$ therefore forces all values on $N(u) \setminus {\mc Q}$ to be integral.  Since $S_1$ and $S_2$ satisfy $S_1 \triangle S_2 = \mc Q$ and $|S_1 \triangle S_2| = 2$, the lemma is proved in this case.  If $|\mc G|$ is even, then the sets $S_1 := \mc G \cup \{ q_1 \}$ and $S_2 := \mc G \cup \{ q_2 \}$ are the only two odd-sized subsets that maximize $\kappa_{u,\vv f}$.  As in the previous case, the maximum value achieved by $\kappa_{u,\vv f}$ over all odd-sized subsets of $N(u)$ is $|N(u)| -1$, so $|\mc E| = 2$ again forces all values on $N(u) \setminus {\mc Q}$ to be integral.  Since $S_1 \triangle S_2 = \mc Q$, we are done in this case as well.

Now suppose that $|\mc E| = 1$.  If $|\mc G|$ is odd, then $S = \mc G$ is the unique maximizer of $\kappa_{u,\vv f}$ over all odd-sized subsets of $N(u)$.  If $|\mc G|$ is even, then $S = \mc G \cup \mc E$ is the unique maximizer of $\kappa_{u,\vv f}$ over all odd-sized subsets of $N(u)$.  In either situation there is at most one active set at $u$, and so the lemma is proved in this case.

Finally, assume that $|\mc E| = 0$.  If $|\mc G|$ is odd, then $S = \mc G$ is the unique maximizer of $\kappa_{u,\vv f}$ over all odd-sized subsets of $N(u)$ and the lemma is proved.  If $|\mc G|$ is even, then the collection of all active sets for $\vv f$ at $u$ is given by $\{ S_q \; | \; q \in {\mc Q}\}$, where $S_q$ is defined as follows:
\[
S_q := \begin{cases}
\mc G \cup\{ q\} & \text{ if } q \not \in \mc G \\
\mc G \setminus \{ q \}&  \text{ if } q \in \mc G.
\end{cases}
\]
If $|\mc Q| = 1$, there is a unique set $S_q$ that maximizes $\kappa_{u,\vv f}(S)$ and we are done.  So suppose that $|\mc Q| > 1$, and let $q_1$ and $q_2$ be distinct elements of $\mc Q$.  By the definition of $\mc Q$, we have $f_{q_1} = f_{q_2}$ if $q_1, q_2 \in \mc L$ or $q_1, q_2 \in \mc G$, and $f_{q_1} = 1-f_{q_2}$ if $q_1 \in \mc L$ and $q_2 \in \mc G$ or vice-versa.  If $q_1, q_2 \in \mc L$ , then
\begin{align*}
|N(u)| - 1 & =  \kappa_{u,\vv f}(S_{q_1}) \\
&  =\sum_{x \in S_{q_1}} f_x + \sum_{x \in N(u) \setminus S{q_1}} (1 - f_x) \\
& = \sum_{x \in S_{q_1} \setminus \{q_1\}} f_x  + f_{q_1}+ (1 - f_{q_2}) \\
& \phantom{==} + \sum_{x \in N(u) \setminus \left(S{q_1} \cup \{q_2\}\right)} (1 - f_x) \\
& = \sum_{x \in S_{q_1} \setminus \{q_1\}} f_x  +1+ \sum_{x \in N(u) \setminus \left(S{q_1} \cup \{q_2\}\right)} (1 - f_x).
\end{align*}
We conclude that $\vv f$ must be integral on $N(u) \setminus \{q_1, q_2\}$;  otherwise, it would not be possible for $\kappa_{u,\vv f}(S_{q_1})$ to attain the value $|N(u)| -1$.  A similar argument yields the same conclusion in each of the other three cases.

By assumption we have that $\vv f$ must be non-integral in at least one position of $N(u)$, so either $0< f_{q_1} < 1$ or $0 < f_{q_2} < 1$.  It follows from this and from the definition of $\mc Q$ that $\mc Q = \{q_1,q_2\}$ and that $0 < f_{q_1}, f_{q_2} < 1$.  Thus, there are exactly two active sets $S_{q_1}$ and $S_{q_2}$, these sets satisfy $|S_{q_1} \triangle S_{q_2}| = 2$, and $S_{q_1} \triangle S_{q_2}$ is exactly the set of indices where $\vv f$ is not integral.
\end{proof}

\subsection{A Proof of Theorem~\ref{thm:halfintegral}} \label{sec:proof}

Let $C$ be a code with parity-check matrix $H$, Tanner graph $G = (X \cup U, E)$, and fundamental polytope $\mc P$.  Let $\vomega \in \mc P$ be an LP pseudocodeword, i.e., an extreme point of $\mc P$. By Theorem~\ref{thm:bfs}, $\vomega$ is a basic feasible solution.  Thus, the span of the constraint vectors of $\mc P$ that are active at $\vomega$ has dimension $n := |X|$.  For each node to which $\vomega$ assigns either a 0 or a 1 there is an active constraint of the form (\ref{eqn:lessthan1}) or (\ref{eqn:greaterthan0}).  This set of constraint vectors is linearly independent, so we may extend it by other active constraint vectors, which will necessarily be of the form $(u,S)$, to obtain a set of $n$ linearly independent constraint vectors that are active at $\vomega$.  Up to a permutation of rows and columns, we may write these $n$ constraint vectors in matrix form as
\[
L_\vomega := \begin{bmatrix}
I_{n-m} & \mathbf 0 \\
Q_\vomega & R_\vomega
\end{bmatrix}
\]
where the last $m$ columns represent the variable nodes in $\mc F_\vomega := \{ x \; | \; 0<\omega_x <1\}$.  Note that since the last $m$ rows of $L_\vomega$ come from constraints of the form $(u,S)$, each entry of $L_\vomega$ is either $-1, 0$, or $+1$.  Note also that since $\text{det}(L_\vomega) = \text{det}(R_\vomega)$, the fact that $L_\vomega$ is invertible implies that the square submatrix $R_\vomega$ is also invertible.

Since each row of $L_\vomega$ is a constraint vector that is active at $\vomega$, we have that $L_\vomega \vomega = \mathbf z_\vomega$, where $\vv z_\vomega$ is an integer vector determined from the right-hand sides of constraints~(\ref{eqn:lessthan1}) -- (\ref{eqn:checkconstraints}).  If we knew that $L_\vomega^{-1}$ takes entries only from $\frac12 \Z$, we could write $\vomega = L_\vomega^{-1} \vv z_\vomega$ and conclude that $\vomega \in \{0, \frac 12, 1\}^n$.  Our next goal is to show that, in the case of cycle codes, $L_\vomega^{-1}$ has entries only in $\frac 12 \Z$.  Because of the block structure of $L_\vomega$, this amounts to showing that $R_\vomega^{-1}$ has entries only in $\frac 12 \Z$.  To show this algebraic fact about $R_\vomega$, we turn to graphical methods.

\begin{defn}\label{defn:fractionalgraph}
Let $C$ be a code with parity-check matrix $H$, Tanner graph $G = (X \cup U,E)$ and fundamental polytope $\mc P$.  Let $\vomega$ be a nontrivial linear programming pseudocodeword, define $\mc F_\vomega :=\{ x \in X \; | \; 0 < \omega_x < 1\}$, and set $m = |\mc F_\vomega|$.  Let $R_\vomega$ be an $m \times m$ matrix formed as above, and define $U_{\vomega}$ to be the set of all check nodes $u \in U$ such that a constraint of the form $(u,S)$ is represented in the rows of $L_\vomega$.  Define the subgraph $G_{R_\vomega}$ of $G$ as follows: let the vertex set of $G_{R_\vomega}$ be $\mc F_\vomega \cup U_{\vomega}$, and make $x \in \mc F_\vomega $ adjacent to $ u \in U_{\vomega}$ if and only if one of the rows of $R_\vomega$ arising from $u$ has a non-zero entry in the $x^\text{th}$ position.
\end{defn}

\begin{lemma}[see also \cite{Fel00}, Theorem 7.1]\label{lemma:fgeq2}
Let $C$ be a code with parity-check matrix $H$, Tanner graph $G = (X \cup U, E)$ and fundamental polytope $\mc P$.  Fix $\vv f \in \mc P$ and $u \in U$.  If $u$ is incident to a variable node $x \in \mc F_{\vv f}:=\{ x \in X \; | \; 0 < f_x < 1\}$, then it is incident to at least two such nodes in $\mc F_{\vv f}$.
\end{lemma}

\begin{proof}
Assume that $u$ is incident to one and only one variable node $x_0 \in \mc F_{\vv f}$.  Then every variable node in $N(u) \setminus\{x_0\}$ is assigned a value of 0 or 1 by $\vv f$.  Let $\mc O_u := \{ x \in N(u) \; | \; f_x = 1\}$.  If $|\mc O_u|$ is even, then $S := \mc O_u \cup \{ x_0\}$ has odd cardinality.  Notice that $\kappa_{u,\vv f}(S) > |S| + |N(u) \setminus S| - 1 =  |N(u)| - 1$, which means that $\vv f$ does not satisfy constraint $(u, S)$ of $\mc P$ given by
\[
\kappa_{u,\vv f}(S) = \sum_{x \in S} f_x + \sum_{x \in N(u)\setminus S}(1- f_x) \leq |N(u)|- 1.
\]
Hence, $\vv f$ is not an element of $\mc P$.

If $|\mc O_u|$ is odd, set $S = \mc O_u$ and observe that $\kappa_{u,\vv f}(S) > |S| + |N(u) \setminus S| - 1 =  |N(u)| - 1$.  Again, $\vv f$ fails to satisfy constraint $(u,S)$ of $\mc P$, so $\vv f$ is not an element of $\mc P$.
\end{proof}

\begin{remark}
Lemma~\ref{lemma:fgeq2} implies that for any $\vv f \in \mc P$, $\mc F_{\vv f}$ is a \emph{stopping set}.  By definition, a stopping set is a set of variable nodes $V$ such that if a check $u$ is adjacent to some $v \in V$, then it is adjacent to at least two distinct elements of $V$.  Stopping sets are significant in the study of iterative message-passing decoding on the binary erasure channel:  the belief propagation algorithm fails to decode if and only if the set of erased bits contains a stopping set~\cite{changyan}.
\end{remark}

We now restrict to the case where $C$ is a cycle code.

\begin{lemma}\label{lemma:2regular}
Let $C$ be a cycle code with Tanner graph $G = (X \cup U, E)$ and fundamental polytope $\mc P$.  Let $L_\vomega$ and $R_\vomega$ be matrices formed as above, and let $\mc F_\vomega$ and $U_\vomega$ be the subsets of $X$ and $U$, respectively, as given in Definition~\ref{defn:fractionalgraph}.  For any nontrivial linear programming pseudocodeword $\vomega \in \mc P$ we have $U_{\vomega} \subseteq N(\mc F_\vomega)$, $|\mc F_\vomega| = |U_{\vomega}|$, and $G_{R_\vomega}$ is 2-regular.
\end{lemma}

\begin{proof}
Let $\mc Q_\vomega$ be given as above, so that
\[
L_\vomega = \begin{bmatrix}
I_{n-m} & \mathbf 0 \\
Q_\vomega & R_\vomega
\end{bmatrix}.
\]
We first show that $U_{\vomega} \subseteq N(\mc F_\vomega)$.  Let $u \in U_{\vomega}$ be given.  There must be a corresponding $S \subseteq N(u)$ such that $(u,S)$ is a row of $L_\vomega$.  Since $R_\vomega$ is non-singular, this row must involve some variable nodes in $\mc F_\vomega$.  Thus, $u \in N(\mc F_\vomega)$, so  $U_{\vomega} \subseteq N(\mc F_\vomega)$.

We now show that $|U_{\vomega}| = |\mc F_\vomega|$.  This amounts to showing that no check is represented in the rows of $[ Q_\vomega \, R_\vomega]$ more than once.  Clearly, a single constraint vector $(u,S)$ associated with $u$ cannot appear in the rows of $[ Q_\vomega \, R_\vomega]$ more than once, since otherwise $R_\vomega$ would not be invertible.

Suppose now that for some $u \in U_\vomega$ there are two distinct subsets $S_1,S_2 \subseteq N(u)$ such that the constraints $(u,S_1)$ and $(u, S_2)$ are both rows of $[ Q_\vomega \, R_\vomega]$.  Then both $S_1$ and $S_2$ are active for $\vomega$ at $u$.  In general, the vectors $(u,S_1)$ and $(u,S_2)$ are such that $(u, S_1) = -(u,S_2)$ when we restrict to those positions of $S_1 \triangle S_2$, and $(u,S_1) = (u,S_2) = \vv 0$ on the positions of $X \setminus N(u)$.  By applying Lemma~\ref{lemma:almostunique} to this check node, we see that $S_1 \triangle S_2 = \{x \in N(u) \; | \; 0 < \omega_x < 1\}$.  This means that $(u,S_1) = - (u,S_2)$ on $S_1 \triangle S_2 = N(u) \cap \mc F_\vomega$ and $(u,S_1) = (u,S_2) = \vv 0$ on $\mc F_\vomega \setminus N(u)$.  It follows that the projections of $(u,S_1)$ and $(u,S_2)$ onto the positions of $\mc F_\vomega$ are scalar multiples of one another, which contradicts in the fact that $R_\vomega$ is invertible.  We conclude that $|U_{\vomega}| = |\mc F_\vomega|$.

To prove the third and final claim of the lemma, we bound the number $e$ of edges in $G_{R_\vomega}$ in two different ways.  Since $C$ is a cycle code, each variable node has degree exactly 2 in $G$, so every variable node in $G_{R_\vomega}$ has degree at most 2.  Thus, $ e \leq 2|\mc F_\vomega|$.  On the other hand, since $U_{\vomega} \subseteq N(\mc F_\vomega)$, Lemma~\ref{lemma:fgeq2} yields $e \geq 2|U_{\vomega}|$.  Since $|\mc F_\vomega| = |U_{\vomega}|$, the result follows.
\end{proof}

\begin{prop}\label{prop:solutionsinhalf}
Let $C$ be a cycle code with fundamental polytope $\mc P$.  For any nontrivial linear programming pseudocodeword $\vomega \in \mc P$ and any corresponding matrix $R_\vomega$, the matrix $R_\vomega^{-1}$ has entries only in $\{-\frac 12, 0, +\frac 12\}$.
\end{prop}

\begin{proof}
By Definition~\ref{defn:fractionalgraph}, the non-zero entries of $R_\vomega$ give the incidence structure of $G_{R_\vomega}$.  More formally, the matrix $|R_\vomega|$ obtained by taking the coordinate-wise absolute value of each entry in $R_\vomega$ is the bipartite incidence matrix for $G_{R_\vomega}$. By Lemma~\ref{lemma:2regular}, $G_{R_\vomega}$ is a 2-regular bipartite graph; therefore, each connected component of $G_{R_\vomega}$ is a cycle of even length.  Using this connection between $R_\vomega$ and $G_{R_\vomega}$, we may assume (by permuting rows and columns) that $R_\vomega$ is a block diagonal matrix
\[
R_\vomega = \begin{bmatrix}
D_1 & \vv 0 & \dots & \vv 0 \\
\vv 0 & D_2 & \dots & \vv 0 \\
\vdots & \vdots & \ddots & \vdots \\
\vv 0 & \vv 0 & \dots & D_b \\
\end{bmatrix}
\]
where each block is square and has the form
\[
D = \begin{bmatrix}
d_{1,1} & d_{1,2} & 0 & \cdots &  0 \\
0 & d_{2,2} & d_{2,3} & \cdots & 0 \\
\vdots &  & \ddots  & \ddots &  \vdots \\
0&  & &  d_{\ell-1,\ell-1}& d_{\ell-1,\ell} \\
d_{\ell,1} & 0  & \dots & 0 &  d_{\ell,\ell} \\
\end{bmatrix},
\]
with $d_{p,q} \in \{-1,+1\}$ for all $p, q$.  To show that $R_\vomega^{-1}$ has entries only in $\{-\frac 12, 0, +\frac 12\}$, it suffices to show that each block $D$ of $R_\vomega$ is such that $D^{-1}$ takes entries only in $\{-\frac 12, 0, +\frac 12\}$.  In fact, we will show that $D^{-1}$ takes entries only in $\{-\frac 12, +\frac 12\}$.

Let $D$ be a block of $R_\vomega$.  Since $R_\vomega$ is invertible, $D$ is also invertible.  This implies the existence of a unique solution $\vv a$ to the equation $D \vv a = {\vepsilon}_p$, where ${\vepsilon}_p$ is the $p$th standard basis vector.  The equation $D \vv a = {\vepsilon}_p$ gives rise to a set of relations that must hold between the entries of $\vv a$.

In the following we take subscripts modulo $\ell$ to respect the cyclic nature of $D$.  For all $q \neq p$, we have $a_q d_{q,q} + a_{q+1}d_{q,q+1} = 0$.  Since $d_{q,q+1}, d_{q,q} \in \{-1,+1\}$ we may rearrange to get $a_q = -\frac{d_{q,q+1}}{d_{q,q}} a_{q+1}$, and so $a_q = \pm a_{q+1}$ for all $q \neq p$.  This in turn implies that all of the entries of $\vv a$ are the same up to sign.

The equation $D \vv a = {\vepsilon}_p$ also implies that $a_p d_{p,p} + a_{p+1} d_{p,p+1} = 1$.  We know from the previous paragraph that $a_{p+1}$ is either $a_p$ or $-a_p$.  Thus, we have to consider two possible equations:   $a_p (d_{p,p} + d_{p,p+1}) = 1$ or $a_p (d_{p,p} - d_{p,p+1}) = 1$.  In either case, since $d_{p,p+1}, d_{p,p} \in \{ -1, +1\}$ we have that $a_p$ is either $-\frac 12$ or $+\frac 12$.  Combining this with the previous paragraph, we see that all the entries of $\vv a$ are in $\{-\frac 12, +\frac 12\}$.

Renaming the unique solution to the equation $D\vv a = \vv {\vepsilon}_p$ to be $\vv a_p$, we see that
\[
D^{-1} = \begin{bmatrix}
| & | & \dots &  | \\
\vv a_1 &\vv a_2 & \dots & \vv a_\ell \\
| & | & \dots &  | \\
\end{bmatrix}
\]
It follows that $D^{-1}$ takes entries only in $\{-\frac 12, \frac 12\}$.
\end{proof}

The next corollary follows immediately from Proposition~\ref{prop:solutionsinhalf} and the discussion surrounding the introduction of $L_\vomega$.

\begin{corollary}\label{cor:inverseinhalves}
Let $C$ be a cycle code with fundamental polytope $\mc P$.  If $\vomega$ is an extreme point of $\mc P$ and $L_\vomega$ is formed from $\vomega$ as above, then $L_\vomega^{-1}$ has entries only in $\frac 12 \Z$.
\end{corollary}

We now have the tools to prove the following proposition.

\begin{prop}\label{prop:verticesinhalves}
Let $C$ be a cycle code of length $n$ with Tanner graph $G = ( X \cup U, E)$ and fundamental polytope $\mc P$.  If $\vomega$ is a nontrivial linear programming pseudocodeword of $\mc P$, then $\vomega \in \{0,\frac 12, 1\}^n$ and the subgraph $G_{\mc F_\vomega}$ of $G$ induced by the set $\mc F_\vomega = \{ x \; | \; \omega_x = \frac 12 \}$ and its neighborhood $N(\mc F_\vomega)$ is precisely $G_{R_\vomega}$, which is 2-regular.
\end{prop}

\begin{proof}
Let $L_\vomega$, $R_\vomega$, and $U_\vomega$ be given as above.  By Corollary~\ref{cor:inverseinhalves}, $L_\vomega^{-1}$ has entries in $\frac 12 \Z$.  By the discussion following the introduction of $L_\vomega$ earlier in this section, we have that $L_\vomega \vomega = \mathbf z_\vomega$, where $\vv z_\vomega$ is an integer vector.  Thus, $\vomega = L^{-1}\mathbf z$ must have entries that come only from $\frac 12 \Z$.  Since any point in the fundamental polytope can only assume values between 0 and 1, $\vomega \in \{0,\frac 12, 1\}^n$.

Lemma~\ref{lemma:2regular} implies that $U_{\vomega} \subseteq N(\mc F_\vomega)$, $|\mc F_\vomega| = |U_{\vomega}|$, and $G_{R_\vomega}$ is 2-regular.  From $U_{\vomega} \subseteq N(\mc F_\vomega)$ and Definition~\ref{defn:fractionalgraph} we see that $G_{R_\vomega}$ is a subgraph of $G_{\mc F_\vomega}$.  Letting $e$ denote the number of edges in $G_{\mc F_\vomega}$, we have that $e = 2|\mc F_\vomega|$ since $C$ is a cycle code.  Lemma~\ref{lemma:fgeq2} implies that $2|N(\mc F_\vomega)| \leq e$.  Using the fact that $U_{\vomega} \subseteq N(\mc F_\vomega)$, we have
\[
2|U_{\vomega}| \leq 2|N(\mc F_\vomega)| \leq e = 2|\mc F_\vomega|.
\]
Since $|\mc F_\vomega| = |U_{\vomega}|$, these inequalities must be tight.  Thus $G_{\mc F_\vomega}$ has the same number of edges as $G_{R_\vomega}$ and contains $G_{R_\vomega}$ as a subgraph.  We conclude that $G_{\mc F_\vomega} = G_{R_\vomega}$.
\end{proof}

We may now prove the main result of this paper.

\begin{theorem}\label{thm:halfintegral}
Let $C$ be a cycle code with Tanner graph $G = (X \cup U, E)$ and fundamental polytope $\mc P$.  If $\vomega \in \mc P$ is a linear programming pseudocodeword of $\mc P$, then the following two conditions hold:
\begin{enumerate}[(a)]
\item $\vomega \in \{0,\frac 12, 1\}^{|X|}$, and

\item with $\mc H:= \{x \in X \, | \, \omega_x = \frac 12\}$, the subgraph $\Gamma$ of $G$ induced by $\mc H \cup N(\mc H)$ is 2-regular.  Equivalently, $\Gamma$ is a union of vertex-disjoint simple cycles $\gamma_1, \gamma_2, \dots, \gamma_\ell$.
\end{enumerate}
\end{theorem}

\begin{proof}
Let $\vomega$ be a linear programming pseudocodeword of $\mc P$.  If $\vomega$ is a codeword, the two conditions in the statement of the theorem are satisfied and we are done.  If, on the other hand, $\vomega$ is a nontrivial linear programming pseudocodeword, we may apply Proposition~\ref{prop:verticesinhalves} to conclude that $\vomega$ satisfies the two conditions.
\end{proof}

\section{Conclusion}\label{sec:conclusion}

We have shown that a vector $\vomega$ can be a linear programming pseudocodeword of a cycle code only if $\vomega$ is half-integral.  Moreover, the subgraph of the Tanner graph induced by the variable nodes assigned a value of $\frac 12$ by $\vomega$, along with their neighboring check nodes, must be a union of vertex-disjoint simple cycles.  These necessary conditions, however, are not sufficient to characterize LP pseudocodewords for cycle codes.  The results from this paper are extended in~\cite{jointpaper2010}, where the author and his collaborator provide complete graphical characterizations of both linear programming pseudocodewords and \emph{minimal linear programming pseudocodewords} (see~\cite{VonKoe06}) for the class of cycle codes.  

Finally, we note that Lemma~\ref{lemma:almostunique}, which was an essential ingredient in our proof of Theorem~\ref{thm:halfintegral}, can be applied to \emph{any} binary linear code.  Testing whether this lemma can shed light on the pseudocodeword structure of additional families of codes is an object of future pursuit.

\section*{Acknowledgement}

The results from this paper are taken from the author's Ph.D. dissertation~\cite{axvig10}.  The author thanks his advisor Judy L. Walker for her support and encouragement.

%


 \bibliographystyle{plain}

\end{document}